\documentclass[reqno,10pt, centertags]{amsart}
\usepackage{amsmath,amsthm,amscd,amssymb}
\usepackage{upref}
\usepackage{latexsym}
\usepackage{lscape}
\usepackage{graphicx}
\usepackage{multirow}
\usepackage{mathrsfs}

\usepackage{upref}
\usepackage{latexsym}
\usepackage{lscape}
\usepackage{graphicx}
\usepackage{tikz}
\usepackage{floatrow}
\usepackage{color}
\usepackage{enumitem}
\usepackage{mathtools}
\usepackage{tikz-cd}
\usepackage{tikz}
\usetikzlibrary{matrix}

\makeatletter
\def\theequation{\@arabic\c@equation}

\newcommand{\gaD}{\gamma_{{}_D}}
\newcommand{\gr}{{\text{graph}}}

\newcommand{\gaN}{\gamma_{{}_N}}

\newcommand{\Sp}{\operatorname{Spec}}
\newcommand{\e}{\hbox{\rm e}}

\newcommand{\bbN}{{\mathbb{N}}}
\newcommand{\bbR}{{\mathbb{R}}}

\newcommand{\C}{{\mathbb{C}}}

\newcommand{\bbC}{{\mathbb{C}}}

\newcommand{\cB}{{\mathcal B}}

\newcommand{\cE}{{\mathcal E}}
\newcommand{\cF}{{\mathcal F}}
\newcommand{\cG}{{\mathcal G}}

\newcommand{\cI}{{\mathcal I}}
\newcommand{\cJ}{{\mathcal J}}
\newcommand{\cK}{{\mathcal K}}
\newcommand{\cL}{{\mathcal L}}

\newcommand{\cP}{{\mathcal P}}
\newcommand{\cQ}{{\mathfrak m}}

\newcommand{\cV}{{\mathcal V}}

\newcommand{\cX}{{\mathcal X}}

\newcommand{\cZ}{{\mathcal Z}}

\newcommand{\bfi}{{\bf i}}

\newcommand{\no}{\nonumber}
\newcommand{\lb}{\label}

\newcommand{\wti}{\widetilde  }

\newcommand{\rank}{\text{\rm{rank}}}
\newcommand{\ran}{\text{\rm{ran}}}

\newcommand{\bi}{\bibitem}

\newcommand{\hatt}{\widehat}

\newcommand{\mi}{\operatorname{Mas}}
\newcommand{\mo}{\operatorname{Mor}}

\numberwithin{equation}{section}

\renewcommand{\det}{\operatorname{det}}

\newcommand{\dom}{\operatorname{dom}}

\newcommand{\tr}{\operatorname{Tr}}

\newcommand{\sign}{\operatorname{sign}}

\newcommand{\spec}{\operatorname{Spec}}
\newcommand{\spflow}{\operatorname{SpFlow}}

\renewcommand{\ker}{\operatorname{ker}}
\newcommand{\diag}{\operatorname{diag}}

\newtheorem{theorem}{Theorem}[section]
\newtheorem{hypothesis}[theorem]{Hypothesis}

\newtheorem{proposition}[theorem]{Proposition}

\theoremstyle{definition}
\newtheorem{definition}[theorem]{Definition}

\newtheorem{remark}[theorem]{Remark}

\newcommand{\dL}{\prescript{d\!}{}L}

\begin{document} 
\begin{abstract}
We show that the spectral flow of a one-parameter family of Schr\"odinger operators on a metric graph is equal to the Maslov index of a path of Lagrangian subspaces describing the vertex conditions. In addition, we derive an Hadamard-type formula for the derivatives of the eigenvalue curves via the Maslov crossing form.
\end{abstract}

\allowdisplaybreaks

\title[Maslov Index]{An index theorem for Schr\"odinger operators on metric graphs}

\author[Y. Latushkin]{Yuri Latushkin}
\address{Department of Mathematics,
The University of Missouri, Columbia, MO 65211, USA}
\email{latushkiny@missouri.edu}
\author[S. Sukhtaiev]{Selim Sukhtaiev	}
\address{Department of Mathematics,
Rice University, Houston, TX 77005, USA}
\email{sukhtaiev@rice.edu}
\date{\today}
\thanks{Supported by the NSF grant  DMS-1710989, by the AMS-Simons Travel Grant, by the Research Board and Research Council of the University of Missouri, and by the Simons Foundation.}
\maketitle

\section{Introduction}\label{sec2}

In this paper we establish a relation between the spectral flow of a one-parameter family of self-adjoint Schr\"odinger operators on a compact metric graph and the Maslov index of a path of Lagrangian subspaces. The spectral flow of a one-parameter family of self-adjoint Fredholm operators is the net number of eigenvalues passing through zero in the positive direction \cite{APS}, \cite{BZ1}. The Maslov index is a topological invariant counting the number of intersections of a curve in the Lagrangian Grassmanian with a fixed cycle \cite{arnold67}, \cite{Arn85}, \cite{BZ1}. It  is a fundamental topological fact that these two quantities are closely related. This relation has been extensively studied in the context of Sturm oscillation theory for systems of differential equations and for multidimensional differential operators, cf., e.g.,  \cite{BbF95}, \cite{BW93}, \cite{BZ3}, \cite{BZ1}, \cite{BZ2}, \cite{CJLS}, \cite{CJM1}, \cite{CJM2}, \cite{DJ11}, \cite{HS18}, \cite{HLS}, \cite{LS17}, \cite{LS1}, \cite{LSS}. In particular, it was recently used to derive an explicit formula for the nodal deficiency of the Dirichlet eigenfunctions \cite{CJM2}, to provide a geometric interpretation of celebrated L.~Friedlander's inequalities \cite{Fr} for the Dirichlet and Neumann eigenvalues \cite{CJM2}, to obtain the oscillation results for Schr\"odinger operators with matrix-valued potentials \cite{HLS}, \cite{HS18}, and to establish instability of pulses in systems of gradient reaction-diffusion equations \cite{BCJLMS}. 

In this work, we consider the Schr\"odinger operator $H=-\frac{d^2}{dx^2}+q$ on a compact metric graph $\Gamma$ subject to the vertex conditions
\begin{equation}\no
Af+Bf'=0,\  f\in\dom(H),
\end{equation}
where $A, B$ are the boundary matrices facilitating self-adjointness of $H$ in $L^2(\Gamma)$. For instance, the Dirichlet boundary condition corresponds to $A=I$, $B=0$, the Robin condition is defined by $A=A^*$, $B=I$. The spectrum of $H$ is discrete and bounded from below, in particular, it accumulates only at $+\infty$. For a family of Schr\"odinger operators $\{H_t\}_{t=0}^1$ corresponding to the matrices of boundary conditions $\{(A_t, B_t)\}_{t=0}^1$, we prove that the spectral flow through zero is equal to the Maslov index of the path of finite dimensional Lagrangian subspaces $\cL_t:=\ran(-B^*_t, A^*_t)$, $t\in[0,1]$,  that is, we derive the formula
\begin{equation}\no
\spflow(\{H_t\}_{t=0}^1)=\mi(\{\cK,\cL_t\}_{t=0}^1),
\end{equation}
where $\cK$ is the Lagrangian subspace formed by the Cauchy data of all solutions to the equation $-f''+qf=0$, see Theorem \ref{thm3.3}. The Maslov index, originally defined as an intersection number, is given by the signature of the Maslov form, a finite dimensional, non-degenerate symmetric form on $\cK\cap\cL_t$ (whenever this intersection is not empty), cf. Theorem \ref{masform}. The signature of the Maslov form is closely related to monotonicity of the eigenvalues passing through zero. An analytical tool furnishing such a connection is given by the Hadamard formula for the derivative of an eigenvalue with respect to the parameter. We establish this formula for Schr\"odinger operators with varying boundary conditions, see Theorem \ref{prop3.4}. Finally, we revisit the classical eigenvalue interlacing inequalities, cf., e.g., \cite[Theorem 3.1.8]{BK}, and derive their modification using the spectral flow formula. 

{\bf Notation.} We denote by $I_n$ the $n\times n$ identity matrix.  For an $n\times m$ matrix $A=(a_{ij})_{i=1,j=1}^{n,m}$
and a $k\times\ell$ matrix $B=(b_{ij})_{i=1,j=1}^{k,\ell}$, we denote by
$A\otimes B$ the Kronecker product, that is, the $nk\times m\ell$ matrix composed of $k\times\ell$ blocks $a_{ij}B$, $i=1,\dots n$, $j=1,\dots m$. We let $\langle\cdot\,,\cdot\rangle_{\cX}$ denote the complex scalar product in the Hilbert space $\cX$. 
We denote by $\cB(\cX)$ the set of linear bounded operators and by $\spec(T)$ the spectrum of an operator $T$ on a Hilbert space $\cX$. Given a subspace $S\subset \cX$ we denote $\prescript{d}{}S:= S\oplus S$. We use notation $J$ for the following $2\times 2$ matrix,
\begin{equation}J:=\left[\begin{matrix}
0& 1\\
-1& 0\\
\end{matrix}\right].
\end{equation}
\section{Preliminaries}
\subsection{Schr\"odinger operators on graphs with fixed edge lengths.} We begin by discussing  differential operators on metric graphs. To set the stage, let us fix a discrete graph $\cG=(\cV,\cE)$ where $\cV$ and $\cE$ denote the set of vertices and edges respectively. We assume that $\cG$ consists of finite number of vertices,  $|\cV|$, and finite number of edges,  $|\cE|$ . Each edge $e\in\cE$  is assigned positive length $\ell_{e}\in(0,\infty)$ and some direction. The corresponding metric graph is denoted by $\Gamma$. The boundary $\partial\Gamma$ of the metric graph is defined by
\begin{equation}
\partial\Gamma:=\cup_{e\in\cE} \{a_e,b_e\}, 	
\end{equation}
where $a_e, b_e$ denote the end points of edge $e$.  It is convenient to treat $2|\cE|$ dimensional vectors as a space of functions of the boundary $\partial\Gamma$, in particular,
\begin{equation}\label{vv1}
L^2(\partial\Gamma)\cong \bbC^{2|\cE|},
\end{equation}
where the space $L^2(\partial\Gamma)=\bigoplus_{e\in\cE}\left( L^2(\{a_e\})\oplus L^2(\{b_e\})\right)$  corresponds to the discrete Dirac measure with support $\cup_{e\in \cE} \{a_e, b_e\}$. In addition to the space of functions on the boundary we consider the Sobolev spaces of functions on the graph $\Gamma$,
\begin{align}
&L^2(\Gamma):=\bigoplus_{e\in\cE}L^2(e),\  \hatt{H}^k(\Gamma):=\bigoplus_{e\in\cE}H^k(e),\ k\in\bbN\no,
\end{align}  
where $H^k(e)$ is the standard $L^2$ based Sobolev space of order $k\in \bbN$.  
As in the case of compact manifolds with boundaries, the spaces $\hatt{L}^2(\Gamma)$ and $L^2(\partial \Gamma)$ are related via the trace maps. We define the Dirichlet and Neumann trace operators  by the formulas
\begin{align}
&\gamma_D: \hatt{H}^2(\Gamma)\rightarrow L^2(\partial \Gamma),
\ \gamma_Df:=f|_{\partial \Gamma}, f\in \hatt{H}^2(\Gamma),\label{vv2}\\
&\gamma_N: \hatt{H}^2(\Gamma)\rightarrow L^2(\partial \Gamma),
\ \gamma_N f:=\partial_{n}f|_{\partial \Gamma}, f\in \hatt{H}^2(\Gamma),\label{vv3}
\end{align}
where $\partial_{n} f$ denotes the derivative of $f$ taken in the inward direction. The trace operator is a bounded, linear operator given by
\begin{equation}\lb{2.4new}
\tr:=
\left[\begin{matrix}
\gamma_D\\
\gamma_N
\end{matrix}\right],\, \tr: \hatt{H}^2(\Gamma)\rightarrow L^2(\partial \Gamma)\oplus L^2(\partial \Gamma)\cong \bbC^{4|\cE|}.
\end{equation}
The Sobolev space of functions vanishing on the boundary $\partial\Gamma$ together with their derivatives is denoted by
\begin{equation}\no
H^2_0(\Gamma):=\left\{f\in \hatt{H}^2(\Gamma): \tr f=0\right\}.
\end{equation}
Using our notation for trace maps, Green's formula can written as follows
\begin{align}\label{vv5}
\int_{\Gamma} f''\overline{g}-f\overline{g''}
&=-\int_{\partial\Gamma}
\partial_{n}f\overline{g}-f\overline{\partial_{n}g} \\
&= -\langle [J\otimes I_{2|\cE|}]\tr f, \tr g\rangle_{\bbC^{4|\cE|}},\  f,g\in  \hatt{H}^2(\Gamma).\no
\end{align}
The right-hand side of  Green's identity defines a symplectic form 
\begin{align} 
\label{vv9}
&\omega:\  \dL^2(\partial \Gamma)\times\,\dL^2(\partial
\Gamma)\rightarrow \bbC, \\
\label{eq:def_omega}
&\omega( (\phi_1, \phi_2), (\psi_1, \psi_2)):=\int_{\partial\Gamma}\phi_2 \overline{\psi_1}-\phi_1\overline{\psi_2},\\
& (\phi_1, \phi_2), (\psi_1, \psi_2)\in \dL^2(\partial\Gamma),\label{vv10}
\end{align}
where $\dL^2(\partial \Gamma):=L^2(\partial \Gamma)\oplus L^2(\partial \Gamma)$.

Next, we introduce the minimal Schr\"odinger operator $H_{min}$ and its adjoint $H_{max}$. To this end, let us fix a bounded real-valued potential $q\in L^{\infty}(\Gamma;\bbR)$. The linear operator 
\begin{equation}\label{b5}
H_{min}:=-\frac{d^2}{dx^2}+q,\quad \dom(H_{min})=\hatt H^2_0(\Gamma),
\end{equation}
is symmetric in $L^2(\Gamma)$. Its adjoint $H_{max}:=H_{min}^*$ is given by the formulas 
\begin{equation}\label{b6}
H_{max}:=-\frac{d^2}{dx^2}+q,\quad \dom(H_{max})=\hatt{H}^2(\Gamma).
\end{equation}
The dificiency indices of $H_{min}$ are finite and equal, that is,
\begin{equation}
0<\dim\ker(H_{max}-\bfi)=\dim\ker(H_{max}+\bfi)<\infty.
\end{equation}
By the standard von-Neumann theory, the self-adjoint extensions of
$H_{min}$ exist and every self-adjoint extension $H$ satisfies
$H_{min}\subset H=H^*\subset H_{max}$. There are various possible
parameterizations of all self-adjoint extensions of the minimal
operator. In this paper we utilize the one stemming from symplectic
geometry \cite{McS}. Namely, we use the fact that the self-adjoint extensions of
the minimal operator are in one-to-one correspondence with the
Lagrangian planes in some symplectic Hilbert space, the fact that goes back to the classical Birman--Vishik--Krein theory \cite{Kr47,Vi}, see also \cite{AS80,BK,BbF95,Ha00,LS1,Pa}. 

A subspace $\cL\subset{}^dL^2(\partial\Gamma):=L^2(\partial\Gamma)\oplus L^2(\partial\Gamma)$ is called {\it Lagrangian} if $\cL$ is equal to its $\omega-$annihilator, i.e., 
\begin{equation}\no
\cL=\cL^{\circ}:=\{x\in L^2(\partial\Gamma): \omega(x,y)=0 \text{\ for all\ }y\in \cL \}.
\end{equation}
The Lagrangian--Grassmannian is the space of Lagrangian planes
\begin{equation}\no
\Lambda(\dL^2(\partial \Gamma)):=\{\cF\subset\dL^2(\partial \Gamma): \cF \ \text{is Lagrangian with respect to \ } \omega \},
\end{equation}
equipped with metric 
\begin{equation}\no
d(\cF_1, \cF_2):=\|P_{\cF_1}-P_{\cF_2}\|_{\cB(\dL^2(\partial \Gamma))},\ \cF_1,\cF_2\in \Lambda(\dL^2(\partial \Gamma)),
\end{equation}
where $P_{\cF}$ denotes the orthogonal projection onto ${\cF}$ in $^dL^2(\partial \Gamma)$.

\begin{proposition} \lb{prop2.1}
	
	i)  \cite{BbF95, Ha00, LS1}
	Assume that $q\in L^{\infty}(\Gamma;\bbR)$. Then the self-adjoint
	extensions of $H_{min}$ $($cf.\ \eqref{b5}$)$ are in one-to-one
	correspondence with the Lagrangian planes  in $\dL^2(\partial \Gamma)$.
	Namely, the following two assertions hold.
	
	1) If $H$ is a self-adjoint extension of $H_{min}$ then 
	\begin{equation}\no
	\cL({H}):=\tr\big({\dom(H)}\big) \text{\   is a Lagrangian plane in\ }\dL^2(\partial \Gamma).
	\end{equation}
	Moreover, the mapping $H\mapsto \cL({H})$ is injective. 
	
	2) Conversely, if $\cL\subset \dL^2(\partial \Gamma)$ is a Lagrangian plane then the operator 
	\begin{align}\label{b7}
	H({\cL}):=-\frac{d^2}{dx^2}+q(x),\ \dom\big(H({\cL})\big)=\{f\in \hatt{H}^2(\Gamma): \tr f\in\cL\},
	\end{align}
	is a self-adjoint extension of $H_{min}$. 
	
	ii) Let $H_{n}, n\geq 0,$ be a sequence of self-adjoint extensions of the operator $H_{min}$ and let $\cL_{n}\subset\dL^2(\partial \Gamma), n\geq 0,$ be the corresponding sequence of Lagrangian planes such that $H_n$ and $\cL_n$  are related to each other as indicated in 1) and 2). Then 
	\begin{equation} \lb{res}
	R(\bfi, H_n)\rightarrow R(\bfi, H_0),\ n\rightarrow\infty, \text{ \ in\  } \cB(L^2(\Gamma)),
	\end{equation}
	$($here $R(\bfi, H_n) $ denotes the resolvent of $H_n$ at $\bfi$ $)$ if and only if
	\begin{equation} \lb{plan}
	\cL_n\rightarrow \cL_0,\ n\rightarrow\infty, \text{ \ in\  }\Lambda(\dL^2(\partial \Gamma)).
	\end{equation}
\end{proposition}
\begin{proof}
This follows from  \cite[Theorem 5.4]{LS1} and the fact that $(\dL^2(\partial \Gamma), \gamma_D, \gamma_N)$ is a boundary triple for the minimal operator $H_{min}$.
\end{proof}

\subsection{The Maslov index of a path of Lagrangian planes in $L^2(\partial \Gamma)\oplus L^2(\partial \Gamma)$} The Maslov index is defined as the spectral flow through the point $1\in \bbC$ of a certain family of unitary matrices, cf. \eqref{ba18}, \eqref{dfnMInd}. This quantity can be expressed in terms of the signature of the crossing form, see \eqref{ab17}. Let us recall the precise definitions from \cite{BZ3}, \cite{BZ1}, \cite{BZ2}.  To that end we introduce the operator 
\begin{equation}\no
\cJ:=\begin{bmatrix}
0_{L^2(\partial\Gamma)}& I_{L^2(\partial\Gamma)}\\
-I_{L^2(\partial\Gamma)}& 0_{L^2(\partial\Gamma)}
\end{bmatrix},
\end{equation}
and notice that the symplectic form $\omega$ defined in \eqref{vv9}--\eqref{vv10} satisfies
\begin{equation}\lb{ab1}
\omega(u,v)=\langle \cJ u,v\rangle_{^d L^2(\partial\Gamma)},\ \ u,v\in\ ^d L^2(\partial\Gamma),
\end{equation}
Furthermore, one has $\cJ^2=-I_{^d L^2(\partial\Gamma)},\ \cJ^{*}=-\cJ$, and
\begin{equation}\lb{ab3}
^d L^2(\partial\Gamma)=\ker(\cJ-\bfi I)\oplus\ker(\cJ+\bfi I).
\end{equation}
Every Lagrangian plane $\cL\subset ^d L^2(\partial\Gamma)$ can be uniquely represented as a graph of a bounded operator $U\in\cB(\ker(\cJ+\bfi I_{^d L^2(\partial\Gamma)}), \ker(\cJ-\bfi I_{^d L^2(\partial\Gamma)}))$, cf. \cite[Lemma 3]{BZ2}, that is,  
\begin{equation}\lb{ab4}
	\cL=\gr(U):=\{y+Uy: y\in \ker(\cJ+\bfi I_{^d L^2(\partial\Gamma)})\}.
\end{equation}
Specifically, for arbitrary $y\in\ker(\cJ+\bfi I_{^d L^2(\partial\Gamma)})$ there exists a unique $z\in \ker(\cJ-\bfi I_{^d L^2(\partial\Gamma)})$ such that $y+z\in\cL$. For such a vector $z$ we set $Uy:=z$. Then 
\begin{equation}\lb{ab5}
	\omega(x,y)=-\omega(Ux,Uy), \ x,y\in \ker(\cJ+\bfi I_{^d L^2(\partial\Gamma)}).
\end{equation}
The operator $U$ is a unitary map  acting between Hilbert spaces $\ker(\cJ+\bfi I_{^d L^2(\partial\Gamma)})$ and $\ker(\cJ-\bfi I_{^d L^2(\partial\Gamma)})$. Indeed, for arbitrary $x,y\in\ker(\cJ+\bfi I_{^d L^2(\partial\Gamma)})$ one has
\begin{align}
\begin{split}
&\langle x,y\rangle_{^d L^2(\partial\Gamma)}=\bfi \langle \cJ x,y\rangle_{^d L^2(\partial\Gamma)}=\bfi\omega(x,y)\lb{ab6}\\
&\quad=-\bfi \omega(Ux,Uy)=-\bfi\langle \cJ Ux,Uy\rangle_{^d L^2(\partial\Gamma)}=\langle Ux,Uy\rangle_{^d L^2(\partial\Gamma)}.
\end{split}
\end{align}

Let us fix a (reference) Lagrangian plane corresponding to a unitary operator $ V\in \cB(\ker(\cJ+\bfi I_{{}^d L^2(\partial\Gamma)}), \ker(\cJ-\bfi I_{^d L^2(\partial\Gamma)}))$, 
\begin{equation}
\cZ\subset{}^d L^2(\partial\Gamma), \cZ=\gr(V).
\end{equation}
In addition, we fix a continuous path
\begin{align}
&\Upsilon:\cI\rightarrow \Lambda(^d L^2(\partial\Gamma)),\ \ \Upsilon(s)=\cF_s,\\
&\Upsilon\in C\big(\cI,  \Lambda(^d L^2(\partial\Gamma))\big), \ \cI=[a,b]\subset\bbR,
\end{align}
and  introduce the corresponding family of unitary operators $U_s$ such that 
\begin{align}
&\hspace{2cm}\cF_s=\gr(U_s),\ s\in\cI,\no\\
&\upsilon : \cI\rightarrow \cB(\ker(\cJ+\bfi I_{^d L^2(\partial\Gamma)}), \ker(\cJ-\bfi I_{^d L^2(\partial\Gamma)})),\  \upsilon(s)=U_s.\no
\end{align} 
The following is proved in \cite{BZ1}:
\begin{align}
& \upsilon\in C(\cI,\cB(\ker(\cJ+\bfi I_{^d L^2(\partial\Gamma)}), \ker(\cJ-\bfi I_{^d L^2(\partial\Gamma)}))),\lb{ab12}\\
& U_sV^{-1}  \text{\ is unitary in\ } \ker(\cJ-\bfi I_{^d L^2(\partial\Gamma)}),\ s\in\cI,\lb{ab13}\\
&\dim (\cF_s\cap \cZ)= \dim \ker (U_sV^{-1}-I_{\cX})\lb{ab15},\ s\in \cI.
\end{align}
Utilizing \eqref{ab12}--\eqref{ab15} we will now define the Maslov index as the spectral flow through the point $1\in\C$ of the family $\upsilon(s), s\in \cI$.  An illuminating discussion of the notion of the spectral flow of a family of closed operators through an admissible curve $\ell\subset\C$ can be found in \cite[Appendix]{BZ2}.  To proceed with the definition, we note that  there exists a partition $a=s_0<s_1<\cdots<s_N=b$ of $[a,b]$ and positive numbers $\varepsilon_j\in(0,\pi)$ such that  $\e^{\pm\bfi \varepsilon_j}\not \in \Sp (U_sV^{-1})$ if $s\in[s_{j-1},s_j]$, for each $1\leq j\leq N$, see \cite[Lemma 3.1]{F04}. Denote
\begin{equation}\lb{ba18}
k(s,\varepsilon):=\sum\nolimits_{0\leq \varkappa\leq \varepsilon}\dim\ker(U_sV^{-1}-\e^{\bfi\varkappa}),\ \varepsilon>0,\ s\in[a,b].
\end{equation}
The Maslov index is defined by the formula
\begin{equation}\lb{dfnMInd}
\text{Mas}(\Upsilon,\cZ):=\sum\limits_{j=1}^{N}\left(k(s_j,\varepsilon_j)-k(s_{j-1},\varepsilon_j)\right).
\end{equation}
We notice that, this definition does not depend on the choice of the partition $\{s_j\}_{j=1}^N$ and $\{\varepsilon_j\}_{j=1}^N$, cf. \cite[Proposition 3.3]{F04}.

Next we turn to the computation of the Maslov index via the crossing forms.  Assume that $\Upsilon\in C^1\big(\cI, \Lambda(^d L^2(\partial\Gamma))\big)$ and let $s_*\in\cI$.  There exists a neighbourhood $\cI_0$ of $s_*$ and a family $R_s\in C^1(\cI_0, \cB(\Upsilon(s_*), \Upsilon(s_*)^{\perp}))$, such that 
\begin{equation}\no
\Upsilon(s)=\{u+R_su\big| u\in \Upsilon(s_*)\},\ s\in \cI_0,
\end{equation}
see, e.g., \cite[Lemma 3.8]{CJLS}.   We will use the following terminology from \cite[Definition 3.20]{F04}.
\begin{definition}\label{def21} Let $\cZ$ be a Lagrangian subspace and $\Upsilon\in C^1\big(\cI, \Lambda(^d L^2(\partial\Gamma))\big)$.
	
	{\it (i)} We call $s_*\in\cI$ a conjugate point or crossing if $\Upsilon(s_*)\cap \cZ\not=\{0\}$.
	
	{\it (ii)} The finite dimentional form $$\mathfrak m_{s_*,\cZ}(u,v):=\frac{d}{ds}\omega(u,R_sv)\big|_{s=s_*}=\omega(u, \dot{R}_{s=s_*}v), \text{\ for\ }u,v \in \Upsilon(s_*)\cap \cZ,$$  is called the crossing form at the crossing $s_*$.
	
	{\it (iii)} The crossing $s_*$ is called regular if the form $\cQ_{s_*,\cZ}$ is non-degenerate, positive if $\cQ_{s_*,\cZ}$ is positive definite, and negative if $\cQ_{s_*,\cZ}$ is negative definite.

\end{definition}

The following result (cf., {\cite[Proposition 3.2.7]{BZ1}}) provides an efficient tool for computing the Malsov index at regular crossings. We denote by $n_+$ and $n_-$ the number of positive and negative squares of a form, the signature is defined by the formula $\sign=n_+-n_-$. 
\begin{theorem} \lb{masform}
	Let $\Upsilon\in C^1\big(\cI, \Lambda(^d L^2(\partial\Gamma))\big)$, and assume that all crossings are regular. Then one has 
	\begin{equation}\lb{ab17}
	\mi\,(\Upsilon,\cZ)=-n_-(\cQ_{a,\cZ})+\sum\limits_{a<s<b}\sign(\cQ_{s,\cZ})+n_+(\cQ_{b,\cZ}).
	\end{equation}
\end{theorem}

We will now review the definition of the Maslov index for {\it two} paths with values in Lagrangian--Grassmannian $\Lambda(^d L^2(\partial\Gamma))$, see \cite[Section 3.5]{F04}. Let us fix two paths of Lagrangian planes
\begin{equation}\no
\Upsilon_1,\Upsilon_2\in C\big(\cI, \Lambda(^d L^2(\partial\Gamma))\big),
\end{equation}
and let $\diag:=\{(p,p):p\in\ ^dL^2(\partial\Gamma)\}$ denote the diagonal plane.  On the Hilbert space $^dL^2(\partial\Gamma)\oplus\ ^dL^2(\partial\Gamma)$ we define the symplectic form $\wti{\omega}:=\omega\oplus(-\omega)$ with the complex structure $\wti{\cJ}:=\cJ\oplus(-\cJ)$ and denote the resulting space of Lagrangian planes  by $\Lambda_{\wti{\omega}}\big(^d L^2(\partial\Gamma)\oplus\, ^dL^2(\partial\Gamma)\big)$. Let
\begin{equation}\no
\wti{\Upsilon}:=\Upsilon_1\oplus\Upsilon_2\in C \big(\cI,\Lambda_{\hat{\omega}}(^d L^2(\partial\Gamma)\oplus\ ^dL^2(\partial\Gamma))\big).
\end{equation}
TheMaslov index of two paths $\Upsilon_1,\Upsilon_2$ is defined by $\mi(\Upsilon_1,\Upsilon_2):=\mi(\wti{\Upsilon},\diag).$
\begin{remark}\lb{rem}
We notice that $\mi(\Upsilon_1, \Upsilon_2)=\mi(\Upsilon_1,\cZ)$ whenever $\Upsilon_2(s)=\cZ$ for all $s\in\cI$ then. If $\Upsilon_1(s)=\cF$ for all $s\in\cI$ then $\mi(\Upsilon_1, \Upsilon_2)=-\mi(\Upsilon_2,\cF)$.
\end{remark}

\section{The spectral flow, the Hadamard-type formula and the Maslov index}
The purpose of this section is twofold: (1) we derive a formula relating the spectral flow of the family of  Schr\"odinger operators and the Maslov index of the associated path of Lagrangian planes; (2) we obtain an Hadamard-type formula relating the derivative of the eigenvalue curves and Maslov crossing form.
\begin{hypothesis}\lb{3.1}
Let $\Upsilon: t\mapsto (A_t, B_t)$  be a one-parameter family  of $2|\cE|\times 4|\cE|$  matrices. Suppose that $\Upsilon\in C^1([\alpha,\beta], \bbC^{2|\cE|\times 4|\cE|})$, $\alpha,\beta \in\bbR$. In addition, suppose that $\rank(A_t, B_t)=2|\cE|$ and $A_tB^*_t=B_tA_t^*$ for all $t$. 
\end{hypothesis}

We refer the reader to \cite[Section 1.4.1]{BK} and \cite{Pa} for the following facts used to describe 
self-adjoint extensions of Schr\"odinger operators on graphs (specifically, see \cite[Lemma 5]{Pa} for item (ii), and the discussion following \cite[Proposition 1]{Pa} for item (iii) below).

\begin{proposition}\lb{rem3.2}
	Assume Hypothesis \ref{3.1}. Let us introduce the following subspace of ${}^d L^2(\partial\Gamma)$,
	\[ \cL_t:=\{(\phi,\psi): A_t\phi+B_t\psi=0\},\, t\in[\alpha,\beta].\]
Then for all $t\in[\alpha,\beta]$ one has
	\begin{itemize}
		\item[(i)] $\cL_t\in\Lambda\big({}^d L^2(\partial\Gamma)\big)$,
		\item [(ii)]  $\cL_t=\{(-B^*_t f,A^*_t f):  f \in L^2(\partial\Gamma)\}$,
		\item [(iii)] $\det (A_tA_t^*-B_tB_t^*)\not=0$,
		\item[(iv)] if $(\phi,\psi)\in\cL_t$ then there is a unique $f\in L^2(\partial\Gamma)$ such that $\phi=-B_t^*f$ and $\psi=A_t^*f$, moreover, $f$ is given by
		 $f=(A_tA_t^*-B_tB_t^*)^{-1}(B_t\phi+A_t\psi)$.
	\end{itemize}
\end{proposition}
In what follows we use the same symbol $\Upsilon$ as in Hypothesis \ref{3.1}
to denote the flow $t\mapsto\cL_t$ of the respective Lagrangian subspaces.

Using the family of matrices $A_t, B_t$ we introduce a family of Schr\"odinger operators as follows
\begin{align}
H_t:&=-\frac{d^2}{dx^2}+q;\ H_t:\dom(H_t)\subset L^2(\Gamma)\rightarrow L^2(\Gamma),\no\\
\dom(H_t)&=\{f\in\hatt{H}^2(\Gamma): A_t\gaD f+B_t\gaN f=0\}\no\\
&=\{f\in\hatt{H}^2(\Gamma):\tr f\in\cL_t\}.\no
\end{align}
By \cite[Theorem 1.4.4, 1.4.19]{BK}, \cite[Proposition 6]{Pa} these operators are self-adjoint extensions of $H_{min}$, their spectra are discrete  and bounded from below, see \cite[Theorem 3.1.1]{BK}. We recall that the number of negative eigenvalues of an operator is called its {\em Morse index}.
Our first goal is to express the difference between the Morse indices of the operators $H_{\alpha}$ and $H_{\beta}$ in terms of the Maslov index of the path of Lagrangian planes $\Upsilon\in C^1\big([\alpha, \beta]; \Lambda({}^dL^2(\partial\Gamma))\big)$. Consequently, we will obtain a relation between the the spectral flow of the family $t\mapsto H_t$ and the Maslov index of $\Upsilon$. Heuristically, the spectral flow is the net number of eigenvalues of $H_t$ that pass through zero in a positive direction as $t$ changes from $\alpha$ to $\beta$. In more rigorous terms, there exists a partition  $\alpha=t_0<t_1<\cdots<t_N=\beta,$ and $N$ intervals $[a_{\ell},b_{\ell}],\ a_{\ell}<0<b_{\ell},\ 1\leq \ell\leq N,$ such that
\begin{equation}\lb{r1}
a_{\ell}, b_{\ell}\not\in\spec\left(H_t\right),\text{\ for all\ }t\in [t_{\ell-1},t_{\ell}], \ 1\leq \ell\leq N.
\end{equation}
Then, the spectral flow through $\lambda=0$ is defined by
\begin{equation}\no
\spflow\left(\{H_t\}_{t=\alpha}^{\beta}\right):=\sum\limits_{\ell=1}^N\sum\limits_{a_{\ell}\leq\lambda<0}\left(\dim\ker\left(H_{t_{\ell-1}}-\lambda\right)-\dim\ker\left(H_{t_{\ell}}-\lambda\right)\right).
\end{equation}
Of course, one can show the spectral flow does not depend on the choice of the partitions, see, for instance, \cite[Appendix]{BZ1}. Moreover, as discussed in \cite[Section 3.3]{LS1}, one has
\begin{equation}\no
\spflow\left(\{H_t\}_{t=\alpha}^{\beta}\right)=\mo(H_{\alpha})-\mo(H_{\beta}).
\end{equation} 

\begin{theorem}\lb{thm3.3}
Assume Hypothesis \ref{3.1} and let 
\begin{equation}\lb{3.3}
\cK_0:=\{(\gaD f, \gaN f):\, f\in\hatt{H}^2(\Gamma) \text{ and } -f''+qf=0\}\in\Lambda\big(^d L^2(\partial\Gamma)\big).
\end{equation}
Then one has
\begin{equation}\lb{3.2}
\mo(H_{\alpha})-\mo(H_{\beta})=\mi(\Upsilon, \cK_0),
\end{equation}
and, consequently, 
\begin{equation}
\spflow(\{H_t\}_{t=\alpha}^{\beta})=\mi(\Upsilon, \cK_0).
\end{equation}
\end{theorem}

\begin{proof} Let us  outline the strategy of the proof. First, we recast the eigenvalue problem $H_t u=\lambda u$ in terms of the intersection of Lagrangian planes
\begin{align}
\begin{split}\lb{3.5}
&\cK_{\lambda}:=\{(\gaD f, \gaN f): \, f\in\hatt{H}^2(\Gamma) \text{ and } -f''+qf=\lambda f\}\in\Lambda\,(^d L^2(\partial\Gamma)),\\
&\cL_t:=\{(u,v)\in\,^dL^2(\partial\Gamma): A_t u+B_tv=0 \}\in\Lambda\,(^d L^2(\partial\Gamma)).
\end{split}
\end{align}
Then we construct a loop of Lagrangian planes $(\cK_{\lambda}, \cL_t)$,  where $(\lambda, t)$ follows the boundary of the square displayed in Figure 1. Due to homotopy invariance, the Maslov index of this loop is equal to zero. Next, we show that the Maslov indices of the parts of the loop corresponding to the horizontal sides of the square are equal to the Morse indices of the respective operators. Finally, using the additivity of the Maslov index under catenation of paths we  obtain \eqref{3.2}.

The operators $H_t$, $\alpha\leq t\leq \beta$ are bounded from below uniformly with respect to $t\in[\alpha,\beta]$, cf., e.g., \cite[Section 3.3]{KS06}. Hence, there exists $\lambda_{\infty}<0$ such that $\ker(H_t-\lambda)=\{0\}$ for all $t\in[\alpha, \beta]$ and all $\lambda\leq \lambda_{\infty}$. For such a $\lambda_{\infty}$ we consider the parameter set $\Sigma$, the square $\cP$ in the $(\lambda,t)$-plane, and the map from $\Sigma$ to $\cP$,  
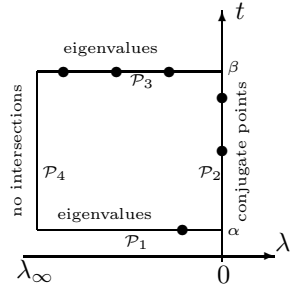
\begin{figure}
	\begin{picture}(100,100)(-20,0)
	\put(78,0){0}
	\put(80,8){\vector(0,1){95}}
	\put(5,10){\vector(1,0){95}}
	\put(71,40){\text{\tiny $\cP_2$}}
	\put(12,40){\text{\tiny $\cP_4$}}
	\put(45,75){\text{\tiny $\cP_3$}}
	\put(43,14){\text{\tiny $\cP_1$}}
	\put(100,12){$\lambda$}
	\put(85,100){$t$}
	\put(80,20){\line(0,1){60}}
	\put(10,20){\line(0,1){60}}
	\put(80,8){\line(0,1){4}}
	\put(2,2){$\lambda_\infty$}
	\put(10,20){\line(1,0){70}}
	\put(10,80){\line(1,0){70}}
	\put(65,20){\circle*{4}}
	\put(80,50){\circle*{4}}
	\put(80,70){\circle*{4}}
	\put(20,80){\circle*{4}}
	\put(40,80){\circle*{4}}
	\put(60,80){\circle*{4}}
	\put(20,87){{\tiny \text{eigenvalues}}}
	\put(14,24){{\tiny \,\,\,\text{eigenvalues}}}  
	\put(85,23){\rotatebox{90}{{\tiny conjugate points}}}
	\put(82,18){{\tiny $\alpha$}}
	\put(82,80){{\tiny $\beta$}}
	\put(0,25){\rotatebox{90}{{\tiny no intersections}}}
	\end{picture}
	\caption{The Morse -- Maslov box: the conjugate points are the eigenvalues}
\end{figure} 
\begin{align}
&\Sigma:=\cup_{j=1}^4\Sigma_j\to\cP=\cup_{j=1}^4\cP_j,\ s\mapsto (\lambda(s), t(s)),\lb{w17w}
\end{align}
where $\cP_j,\ j=1,\cdots, 4$ are the sides of the positively oriented boundary of the square $[\lambda_{\infty},0]\times [\alpha,\beta]$, and the parameter set $\Sigma=\cup_{j=1}^4 \Sigma_j$ and $\lambda(\cdot)$, $t(\cdot)$ are defined as follows:
\begin{align}
&\lambda(s)=s,\, t(s)=\alpha ,\, s\in\Sigma_1:=[\lambda_{\infty},0],\lb{18}\\
&\lambda(s)=0,\, t(s)=s+\alpha ,\, s\in\Sigma_2:=[0,\beta-\alpha ],\lb{19}\\
&\lambda(s)= -s+\beta-\alpha ,\, t(s)= \beta,\, s\in\Sigma_3:=[\beta-\alpha,\beta-\alpha -\lambda_{\infty}],\lb{20}\\
&\lambda(s)=\lambda_{\infty},\, t(s)=-s+2\beta-\alpha -\lambda_{\infty},\lb{21}\\
&\hskip3cm s\in\Sigma_4:=[\beta-\alpha -\lambda_{\infty}, 2(\beta-\alpha)-\lambda_{\infty}].\no
\end{align} 

The mapping
\begin{equation}\no
\tr: \ker\big(H_{t(s)}-\lambda(s)\big)\rightarrow \cK_{\lambda(s)}\cap\cL_{t(s)},\ s\in\Sigma,
\end{equation}
is one-to-one and onto, hence, 
\begin{equation}\lb{3.11}
\dim (\ker(H_{t(s)}-\lambda(s)))=\dim ( \cK_{\lambda(s)}\cap\cL_{t(s)}),\ s\in\Sigma.
\end{equation}
In particular, $\lambda(s)$ is an eigenvalue of $H_s$ if and only if $ \cK_{\lambda(s)}\cap\cL_{t(s)}\not=\{0\}$. Using this observation we will first show that
\begin{equation}\lb{3.12}
\mi(\cK_{\lambda(s)}|_{\Sigma_1}, \cL_{\alpha})=-\mo(H_{\alpha}) \text{\  and\ } \mi(\cK_{\lambda(s)}|_{\Sigma_3}, \cL_{\beta})=\mo(H_{\beta}).
\end{equation}
The argument is based on a standard computation of the Maslov form at the crossings on the horizontal sides of the square, cf., e.g., \cite[(5.3)]{BbF95}.  Let us focus on the first equality in \eqref{3.12}, the proof of the second one is analogous. We will show  that each crossing on $\Sigma_1$ is negative (hence, non-degenerate), and use \eqref{ab17} to verify that geometric multiplicities of negative eigenvalues of $H_{\alpha}$ add up to minus the Maslov index. 
To begin the proof of the first equity in \eqref{3.12}, we let $s_*\in [\lambda_{\infty}, 0]$ be a conjugate point so that $\cK_{\lambda(s_*)}\cap\cL_{\alpha}\not=\{0\}$. By \cite[Theorem 3.8 and Remark 3.9]{BbF95} the map  $s\mapsto \cK_{\lambda(s)}$ is contained in $C^1\big([\lambda_{\infty},0], \Lambda(^d L^2(\partial\Gamma))\big)$. Then there exists a small neighbourhood  $\Sigma_{s_*}\subset[\lambda_{\infty}, 0]$ of $s_*$  and a family of operators $R_{s+s_*}$ so that
\begin{equation}\lb{w35}
(s+s_*)\mapsto R_{(s+s_*)}\ \text{in}\  C^1\big(\Sigma_{s_*}, \cB(\cK_{\lambda(s_*)}, (\cK_{\lambda(s_*)})^{\perp})\big),\ R_{s_*}=0,
\end{equation} 
and 
\begin{equation}
\cK_{\lambda(s)}=\{(\phi,\psi)+R_{s+s_*}(\phi,\psi)\big| (\phi,\psi)\in \cK_{\lambda(s_*)} \}\ \text{for all }\ (s+s_*)\in \Sigma_{s_*},
\end{equation}
see, e.g.,   \cite[Lemma 3.8]{CJLS}. 
Let us fix $(\phi_0,\psi_0)\in \cK_{\lambda(s_*)}$ and consider the family 
\begin{equation*}
(\phi_s,\psi_s):=(\phi_0,\psi_0)+R_{(s+s_*)}(\phi_0,\psi_0)\text{\ with small\ }|s|.
\end{equation*}
Since $(\phi_s,\psi_s)\in\cK_{\lambda(s)}$, there exists a unique $u_s$ satisfying 
\begin{equation*}
-u_s''+qu_s=\lambda(s+s_*)u_s\text{\ and\ }\tr u_s=(\phi_s,\psi_s)\text{\ for small\ }|s|.
\end{equation*}
Next, using \eqref{vv5} we calculate:
\begin{align*}
&\omega\left((\phi_0,\psi_0), (\phi_0,\psi_0)+ R_{(s+s_*)}(\phi_0,\psi_0)\right)=\int_{\partial\Gamma}\psi_0 \overline{\phi_s}-\phi_0\overline{\psi_s}\no\\
&\quad=-\int_\Gamma u_0''\overline{u_s}-u_0\overline{u''_s}\no\\
&\quad=\langle   -u''_0+qu_0, u_s\rangle_{L^2(\Gamma)}- \langle u_0,- u''_s+qu_s\rangle_{L^2(\Gamma)}\no\\
&\quad=\langle  \lambda(s_*) u_0, u_s\rangle_{L^2(\Gamma)}- \langle u_0,\lambda(s_*+s) u_s\rangle_{L^2(\Gamma)}
=-\langle u_0,s u_s\rangle_{L^2(\Gamma)}.\
\end{align*}
Recalling Definition \ref{def21} {\it (ii)}, we evaluate the crossing form
\begin{align}
&\cQ_{s_*,\cL_{\alpha}}\left((\phi_0,\psi_0),(\phi_0,\psi_0)\right):=\frac{d}{ds}\omega\left((\phi_0,\psi_0), R_{(s+s_*)}(\phi_0,\psi_0)\right)\big|_{s=0}\no\\
&=\lim\limits_{s\rightarrow 0} \frac{\omega\left((\phi_0,\psi_0), R_{(s+s_*)}(\phi_0,\psi_0)\right)}{s}=\lim\limits_{s\rightarrow 0} \frac{-\langle u_0,s u_s\rangle_{L^2(\Gamma)}}{s}=-\|u_0\|^2_{L^2(\Gamma)},\no
\end{align}
where we used the continuity of $s\mapsto u_s$ at $0$ established in a more general setting in \cite[page 355]{LSS}.
Therefore, the crossing form is negative definite at all conjugate points on $[\lambda_{\infty}, 0]$ and, using \eqref{ab17}, one obtains
\begin{align}
&\mi\left(\cK_{\lambda(s)}|_{s\in\Sigma_1},\cL_{\alpha}\right)=-n_-\left(\cQ_{\lambda_{\infty},\cL_{\alpha}}\right)+\sum\limits_{\substack{\lambda_{\infty}<s<0:\\
		\cK_{\lambda(s)}\cap \cL_{\alpha}\not=\{0\}	}}\sign\ \cQ_{s,\cL_{\alpha}}\no\\
&\quad+n_+(\cQ_{0,\cL_{\alpha}})=-\sum\limits_{\lambda_{\infty}\leq s< 0}\dim \ker\left(H_{\alpha}-\lambda(s)\right)=-\mo\left(H_{\alpha}\right),\lb{w38}
\end{align}
where we used $n_+\left(\cQ_{0,\cL_{\alpha}}\right)=0$, and the fact that there are no crossings to the left of $\lambda_{\infty}$.

By the additivity of the Maslov index under catenation of paths we get
\begin{align}
\begin{split}
&\mi\left((\cK_{t(s)},\cL_{t(s)})|_{s\in\Sigma}\right)=\mi\left((\cK_{t(s)},\cL_{t(s)})|_{s\in\Sigma_1}\right)\lb{w42}\\
&\quad+\mi\left((\cK_{t(s)},\cL_{t(s)})|_{s\in\Sigma_2}\right)+\mi\left((\cK_{t(s)},\cL_{t(s)})|_{s\in\Sigma_3}\right)\\
&\quad+\mi\left((\cK_{t(s)},\cL_{t(s)})|_{s\in\Sigma_4}\right).
\end{split}
\end{align} 
Finally, using $\mi\left((\cK_{t(s)},\cL_{t(s)})|_{s\in\Sigma}\right)=0$ (by homotopy invariance) and $\mi\left((\cK_{t(s)},\cL_{t(s)})|_{s\in\Sigma_4}\right)=0$ (since there are no crossing on $\cP_4$), we arrive at \eqref{3.2}.
\end{proof}
The following result provides an Hadamard-type formula for the derivative of the eigenvalue curves of the operator family $H_t$, $\alpha\leq t\leq \beta$. Formulas of this type have rich history that goes back to \cite{H} and \cite{GS}; further information can be found in \cite{BLC,Gr,Henry} and \cite{LS17}.
The dependence of the eigenvalues of $H_t$ on boundary matrices $(A_t,B_t)$ is discussed in \cite[Theorems 3.1.2 and 3.1.4]{BK}. In particular, it is known from these results that simple eigenvalues and the family of respective eigenfunctions are differentiable with respect to the parameter $t$.
\begin{theorem}\lb{prop3.4}
Assume Hypothesis \ref{3.1} and fix $t_0\in(\alpha, \beta)$. Suppose that $\lambda_{t_0}$ is a simple eigenvalue of $H_{t_0}$ and let  $u_{t_0}$ be the normalized eigenfunction.  Then \begin{equation}\label{3.17}
\frac{d\lambda_t}{dt}\Big|_{t=t_0}= \big\langle (A_{t_0}\dot{B}_{t_0}^*-B_{t_0}\dot{A}_{t_0}^*) \phi _{t_0}, \phi_{t_0}\big\rangle_{L^2(\partial\Gamma)}=\cQ_{t_0,\cK_{\lambda_{t_0}}}(\tr u_{t_0}, \tr u_{t_0}),
\end{equation}
where $\phi_{t_0}:=(A_{t_0}A_{t_0}^*-B_{t_0}B_{t_0}^*)^{-1}(B_{t_0}\gaD u_{t_0}+A_{t_0}\gaN u_{t_0})$ and the derivative with respect to $t$ is denoted by ``dot".
\end{theorem}
\begin{proof}
First we compute the derivative of the eigenvalue curve $\lambda_t$.  Since the vector valued function $t\mapsto u_t$ is differentiable near $t_0$ by \cite[Theorem 3.1.2 and 3.1.4]{BK}, we may differentiate the eigenvalue equation $H_tu_t=\lambda_tu_t$ for $t$ sufficiently close to $t_0$, thus obtaining
\begin{equation}\lb{3.18}
-\dot{u}_t''+q\dot{u}_t=\dot{\lambda}_tu_t+\lambda_t\dot{u}_t.
\end{equation}
Next, taking the scalar product of both sides of this equation with $u_t$ yields
\begin{equation}
\langle-\dot{u}_t'',u_t\rangle_{L^2(\Gamma)}+\langle q\dot{u}_t,u_t\rangle_{L^2(\Gamma)}=\dot{\lambda}_t+\lambda_t\langle\dot{u}_t,u_t\rangle_{L^2(\Gamma)}.\no
\end{equation}
Green's formula \eqref{vv5} and \eqref{eq:def_omega} imply
\begin{align}
\dot{\lambda}_t=\langle \dot{u}_t, Hu_t\rangle_{L^2(\Gamma)}-\lambda_t\langle\dot{u}_t,u_t\rangle_{L^2(\Gamma)}+\omega(\tr\dot{u}_t,\tr u_t),
\end{align}
and since $Hu_t=\lambda_tu_t$ we have
\begin{equation}\label{lambdadot}
\dot{\lambda}_t=\omega(\tr\dot{u}_t,\tr u_t).
\end{equation}
Since $\tr u_t=(\gaD u_t, \gaN u_t)\in\cL_t$, by Proposition \ref{rem3.2}(iii) there exists a unique $\phi_t\in L^2(\partial\Gamma)$ such that 
\begin{equation}\lb{3.23}
\tr u_t=(-B_t^*\phi_t, A_t^*\phi_t).
\end{equation}
Solving this equation for $\phi_t$ we have
\begin{equation}\lb{3.23a}
\phi_t=(A_{t}A_{t}^*-B_{t}B_{t}^*)^{-1}(B_{t}\gaD u_{t}+A_{t}\gaN u_{t}),
\end{equation}
and thus the mapping $t\mapsto \phi_t$ is differentiable. Differentiating \eqref{3.23} we obtain 
\begin{equation}\lb{3.24}
\tr\dot{u}_t=\big(-\dot{B}_t^*\phi_t, \dot{A}_t^*\phi_t\big)+\big(-B_t^*\dot{\phi}_t,
A_t^*\dot{\phi}_t\big).
\end{equation}
Plugging this and \eqref{3.23} in \eqref{lambdadot}, using that $\ran(-B_t^*, A_t^*)$ is a Lagrangian plane by Proposition \ref{rem3.2}(ii) and formula \eqref{eq:def_omega} for the symplectic form, we have
\begin{align}\label{lddot}
\frac{d\lambda_t}{dt}\Big|_{t=t_0}&=\omega\big((-\dot{B}_{t_0}^*\phi_{t_0}, \dot{A}_{t_0}^*\phi_{t_0}),
(-B_{t_0}^*\phi_{t_0}, A_{t_0}^*\phi_{t_0})\big)\\
&=\big\langle\dot{A}_{t_0}^*\phi_{t_0},-B_{t_0}^*\phi_{t_0} \big\rangle_{L^2(\partial\Gamma)}-\big\langle-\dot{B}_{t_0}^*\phi_{t_0} , A_{t_0}^*\phi_{t_0}\big\rangle_{L^2(\partial\Gamma)}\nonumber\\
&=\big\langle (A_{t_0}\dot{B}_{t_0}^*-B_{t_0}\dot{A}_{t_0}^*) \phi _{t_0}, \phi_{t_0}\big\rangle_{L^2(\partial\Gamma)},\nonumber
\end{align}
thus completing the proof of the first equality in \eqref{3.17}

Next, we compute the Maslov crossing form. Since $\lambda_{t_0}$ is an eigenvalue of $H_{t_0}$, the point $t_0\in[\alpha,\beta]$ is the conjugate point for the path $\cL_t$ with respect to a reference plane $\cK_{\lambda_{t_0}}$, i.e. $\cL_{t_0}\cap \cK_{\lambda_{t_0}}\not=\{0\}$. Since the map $t\mapsto \cL_t$ is contained in $C^1\big([\alpha,\beta], \Lambda({}^dL^2(\partial\Gamma))\big)$, by \cite[Lemma 3.8]{CJLS} there exists a small neighbourhood  $\Sigma_{t_0}\subset(\alpha,\beta)$ of $t_0$  and a family of operators $R_{t}$ so that the map
\begin{equation}
t\mapsto R_{t}\ \text{is in}\  C^1\big(\Sigma_{t_0}, \cB(\cL_{t_0}, \cL_{t_0}^{\perp})\big),\ R_{t_0}=0,
\end{equation} 
and 
\begin{equation}\lb{w36}
\cL_{t}=\{v+R_{t}v\big| v\in \cL_{t_0} \}\ \text{for all }\ t\in \Sigma_{t_0},
\end{equation}
Let  $v_{t_0}:=\tr u_{t_0}\in \cL_{t_0}$ and consider the family 
\begin{align}
v_t:&=v_{t_0}+R_{t}v_{t_0}\in \cL_{t}\subset {}^dL^2(\partial\Gamma),\ t\in \Sigma_{t_0}.
\end{align}
By definition of the crossing form
\begin{equation}\label{omegav}
\mathfrak{m}_{t_0,K_{\lambda_{t_0}}}(v_{t_0},v_{t_0})=-\frac{d}{dt}
\omega(v_{t_0},v_t)\big|_{t=t_0}=-\omega(v_{t_0},\dot{v}_{t_0}).
\end{equation}
Let us notice that the minus sign in \eqref{omegav} comes from the definition of the Maslov index for two paths as discussed after Theorem \ref{masform}, see Remark \ref{rem}.
Since $v_t\in\cL_{t}$ by construction, due to Proposition \ref{rem3.2} for  $t\in \Sigma_{t_0}$ there exists a unique $f_t\in L^2(\partial\Gamma)$ such that
\begin{equation}\lb{3.31}
v_t=(-B_{t}^*f_{t}, A_t^*f_{t}),
\end{equation}
moreover,
$f_t:=(A_{t}^*A_{t}-B_{t}^*B_{t})^{-1}(B_{t}p_t+A_{t}q_t)$,
where we split $v_t=(p_{t},q_{t})\in L^2(\partial\Gamma)\oplus L^2(\partial\Gamma)$. Therefore the mapping $t\mapsto f_t$ is differentiable in $\Sigma_{t_0}$. Differentiating  \eqref{3.31} yields
\begin{equation}\label{vdot}
\dot{v}=\big(-\dot{B}_t^* f_t,\dot{A}_t^* f_t\big)+\big(-B_t^*\dot{f}_t
,A_t^*\dot{f}_t\big).
\end{equation}
Note that $f_{t_0}=\phi_{t_0}$ due to the uniqueness of the representations \eqref{3.31} and \eqref{3.23}, and because $v_{t_0}=\tr u_{t_0}$.
Plugging \eqref{vdot} and \eqref{3.31} into \eqref{omegav} and using that $\ran(-B_t^*, A_t^*)$ is a Lagrangian plane by Proposition \ref{rem3.2}(ii),  we have
\begin{align*}
\mathfrak{m}_{t_0,K_{\lambda_{t_0}}}(v_{t_0},v_{t_0})&=
-\omega\big((-B_{t_0}^*f_{t_0}, A_{t_0}^*f_{t_0}), (-\dot{B}_t^* f_t,\dot{A}_t^* f_t)\big)\\
&= \omega\big((-\dot{B}_{t_0}^* \phi_{t_0},\dot{A}_{t_0}^* \phi_{t_0}), (-B_{t_0}^*\phi_{t_0}, A_{t_0}^*\phi_{t_0})\big)=\frac{d\lambda_t}{dt}\Big|_{t=t_0},
\end{align*}
where in the last equality we used  \eqref{lddot}. \end{proof}
\begin{remark}
Our assumption about simplicity of $\lambda_{t_0}$ may be removed. If $d:=\dim(\cK_{\lambda_{t_0}})>1$ then $d$ eigenvalue curves cross at $t_0$. An Hadamard-type formula \eqref{3.17} for each of these curves is still valid with $\phi_{t_0}$ replaced by the corresponding normalized basis vector  of $\cK_{\lambda_{t_0}}\cap \cL_{t_0}$. Of course, in this case the  eigenvectors are not necessarily differentiable with respect to $t$. Hence,  \eqref{3.18} cannot be used and an alternative argument is required. Such argument based on analytic perturbation theory was carried out in \cite{LS17} in a different context. 
\end{remark}

To demonstrate an application Theorem \ref{thm3.3} and Theorem \ref{prop3.4}, we discuss a well-known eigenvalue interlacing result for quantum graphs, cf.\ \cite[Theorem 3.1.8]{BK}. Consider the Schr\"odinger operator $H_{t}=-\frac{d^2}{dx^2}+q$ on a star graph $\Gamma$ with a bounded real-valued potential subject to arbitrary self-adjoint vertex conditions at the vertices of degree one, and the following $\delta$-type conditions at the center $v$,
	\begin{equation}\label{BCnu}
	\sum_{e\sim v}{\partial_{n} f_e(v)}=t f(v),\ t\in\bbR,
	\end{equation}
In this case the  boundary matrices describing the vertex conditions (cf. Proposition \ref{rem3.2}) are given by $\wti A\oplus A_{t}$ and $\wti B\oplus B$ where 
	\begin{equation}\no
	A_{t}= \begin{bmatrix}
	1&-1&&\cdots &0\\
	0&1&-1&\cdots&0\\
	&&\ddots\\
    0&&&1&-1\\
	-t&0&\cdots&&0
	\end{bmatrix},\quad  
	B= \begin{bmatrix}
	0&&\cdots &&0\\
	0&&\cdots&&0\\
	&&\ddots\\
	0&&\cdots &&0\\
	1&1&\cdots&&1
	
	\end{bmatrix},
	\end{equation}
	and the matrices $\wti A$ and $\wti B$ correspond to the vertex conditions at $\cV\setminus\{v\}$. Clearly, one has
	\begin{equation}\lb{3.33}
	A_{t}^*B=B^*A_{t}= \begin{bmatrix}
	0&0&\cdots &&0\\
	0&0&\cdots&&0\\
	&&\ddots\\
	0&0&\cdots&&-t
	\end{bmatrix}.\  
	\end{equation}

For $t\in\bbR$, let $\lambda_n(t)$ denote the $n$-th eigenvalue of the Schr\"odinger operator $H_t$ subject to $\delta-$type condition \eqref{BCnu}, and let $\phi_{n, t}$ denote the corresponding eigenfunction of $H_t$. Next, we provide a modification of the classical interlacing inequalities for the eigenvalues of $H_t$ cf., e.g., \cite[Theorem 3.1.8]{BK}, and prove it using the spectral flow formula.
\begin{proposition}
Fix $\nu\in\bbR$ and $n\in\bbN$. Assume that $(\lambda_n(\nu), \phi_{n,\nu})$ is  a simple eigenpair of $H_{\nu}$ and suppose that $\phi_{n, \nu}(v)\not=0$. 
Then for arbitrary $\mu\in\bbR$ and $\theta\in\bbR$ one has
\begin{align}
\begin{split}\lb{3.32new}
&\lambda_{n-1}(\mu)<\lambda_n(\nu)< \lambda_{n+1}(\theta),
\end{split}
\end{align}
In addition, the function $t\mapsto \lambda_{n}(t)$ is strictly monotonically increasing near $\nu$. 
\end{proposition}
\begin{proof}
First, we notice that \eqref{3.17}, \eqref{3.33}  $\nu$ yield
\begin{equation}\lb{3.31new}
\lambda'_n(\nu)=|\phi_{n, \nu}(v)|>0.
\end{equation} 
Hence, the function $t\mapsto \lambda_{n}(t)$ is strictly monotone near $\nu$.

Heuristically, \eqref{3.32new} follows from the fact that the spectral flow through $\lambda_n(\nu)$ is equal to one. That is, the families of eigenvalues $\{\lambda_{n\pm1}(t)\}_{t\in\bbR}$ do not cross $\lambda_n(\nu)$. 

Let us now provide a rigorous proof. First, we claim that $\nu$ is a unique crossing point on the line $\lambda=\lambda_{n}(\nu)$, that is, 
\begin{equation}\lb{3.35}
\lambda_n(\nu)\not \in \spec( H_t),\  t\not=\nu.
\end{equation}
Seeking a contradiction, we assume  that $\lambda_n(\nu)=\lambda_k(\tau) \in \spec( H_{\tau})$ for some $\tau\not=\nu$, $k\in\bbN$, and denote the corresponding eigenfunction by $\phi_{k,\tau}$. We will show that $\displaystyle (\tau+\nu)/2$ is a also a crossing, in other words, 
\begin{equation}\lb{3.36}
\lambda_n(\nu)\in\spec\left(H_{\frac{\tau+\nu}{2}}\right).
\end{equation}
To that end we define a function
\begin{equation}\no
\Phi:= \frac{1}{2}\left(\Big(\frac{\phi_{k,\tau}(v)}{\phi_{n,\nu}(v)}\Big)\phi_{n,\nu}+\phi_{k,\tau} \right),
\end{equation}
and notice that
\begin{equation}
\sum_{e\sim v}{\partial_{n} \Phi(v)}=\frac{\tau+\nu}2 \Phi(v).
\end{equation}
In addition, since $\lambda_n(\nu)=\lambda_k(\tau)$, one has $-\Phi''+q\Phi=\lambda_n(\nu) \Phi$. Thus \eqref{3.36} holds true.
Repeating this procedure one can produce a sequence of positive crossings converging to $\nu$. However, existence of such a sequence contradicts the fact that $\nu$ is a regular crossing (cf. \cite[Corollary 3.25]{F04}). Hence, $\nu$ is a unique crossing on the line $\lambda=\lambda_{n}(\nu)$ as asserted. 

Let us fix an arbitrary $\varkappa>0$ and recall 
the Lagrangian planes $\cL_t, \cK_{\lambda}$ from \eqref{3.5}. Then  Theorem \ref{thm3.3} yields
\begin{equation}\lb{3.38}
\mo(H_{\nu}-\lambda_{n}(\nu))-\mo(H_{\nu+\varkappa}-\lambda_{n}(\nu))=\mi(\{\cL_t, \cK_{\lambda_{n}(\nu)}\}_{t=\nu}^{\nu+\varkappa}).
\end{equation}
Since $\nu$ is a positive crossing it does not contribute to the Maslov index of the path $\{\cL_t, \cK_{\lambda_{n}(\nu)}\}_{t=\nu}^{\nu+\varkappa}$ according to \eqref{ab17}. Furthermore, as we have shown earlier $\nu$ is a unique crossing, thus the Maslov index of this path is equal to zero. That is, combining  \eqref{ab17}, \eqref{3.17}, \eqref{3.31new} we obtain
\begin{equation}\lb{3.39}
\mi(\cL_t, \cK_{\lambda_{n}(\nu)})_{t=\nu}^{\nu+\varkappa}=0.
\end{equation}
Next, $\mo(H_{\nu}-\lambda_{n}(\nu))=n-1$ since $\lambda_{n}(\nu)$ is the $n-th$ eigenvalue of $H_{\nu}$. Thus by \eqref{3.38}, \eqref{3.39} one has
\begin{equation}\lb{3.40}
\mo(H_{\nu+\varkappa}-\lambda_{n}(\nu))=n-1,
\end{equation}
Similarly, using \eqref{ab17}, \eqref{3.17}, \eqref{3.31new} we compute the Maslov index of the path $\{\cL_t,\cK_{\lambda_{n}(\nu)}\}_{t=\nu-\varkappa}^{\nu}$ and the corresponding Morse indices as follows
\begin{align}
\mi(\cL_t, \cK_{\lambda_{n}(\nu)})_{t=\nu-\varkappa}^{\nu}&=1,\no \\
\mo(H_{\nu-\varkappa}-\lambda_{n}(\nu))-\mo(H_{\nu}-\lambda_{n}(\nu))&=\mi(\{\cL_t, \cK_{\lambda_{n}(\nu)}\}_{t=\nu-\varkappa}^{\nu}),
\end{align}
hence, 
\begin{equation}\lb{3.42}
\mo(H_{\nu-\varkappa}-\lambda_{n}(\nu))=n.
\end{equation}
To summarize, \eqref{3.40} and \eqref{3.42} yield
\begin{align}\lb{3.43}
&\#\{j\in\bbN: \lambda_{j}(t)< \lambda_{n}(\nu)\}\in\{n-1, n\}, \text{\ for all\ }t\in\bbR.
\end{align}
We are now ready to prove \eqref{3.32new}. Suppose that $\lambda_{n-1}(\mu)\geq\lambda_n(\nu)$ for some $\mu\in\bbR$. Then  $\lambda_{n-1}(\mu)\not=\lambda_n(\nu)$ by \eqref{3.35}, hence, 
\begin{equation}\no
\#\{j\in\bbN: \lambda_{j}(\mu)< \lambda_{n}(\nu)\}\leq n-2,
\end{equation}
which contradicts \eqref{3.43}. Likewise, assuming that $\lambda_n(\nu)\geq \lambda_{n+1}(\theta)$ for some $\theta\in\bbR$ we arrive at 
\begin{equation}\no
\#\{j\in\bbN: \lambda_{j}(\theta)< \lambda_{n}(\nu)\}\geq n+1,
\end{equation}
which again contradicts \eqref{3.43}.
\end{proof}

\end{document}